\documentclass[times]{amsart}
\usepackage{latexsym}
\usepackage{graphics}
\usepackage{amssymb}
\usepackage{amsmath}
\usepackage{amscd}
\usepackage{amsthm}
\usepackage{tikz}
\newcommand{\Aut}{\textrm{Aut}}

\newtheorem{theorem}{Theorem}[section]
\newtheorem{lemma}[theorem]{Lemma}
\newtheorem{proposition}[theorem]{Proposition}

\newtheorem{example}[theorem]{Example}
\newtheorem{definition}[theorem]{Definition}

\begin{document}

\title{Genus Bounds for Harmonic Group Actions on Finite Graphs}
\author{Scott Corry}

\address{Lawrence University, Department of Mathematics, 711 E. Boldt Way -- SPC 24, Appleton, WI 54911, USA}
\email{corrys@lawrence.edu}
\thanks{A portion of this research was carried out while the author was a Visiting Fellow at the Isaac Newton Institute for Mathematical Sciences in Cambridge, UK}

\maketitle

\begin{abstract}
This paper develops graph analogues of the genus bounds for the maximal size of an automorphism group of a compact Riemann surface of genus $g\ge 2$. Inspired by the work of M. Baker and S. Norine on harmonic morphisms between finite graphs, we motivate and define the notion of a harmonic group action. Denoting by $M(g)$ the maximal size of such a harmonic group action on a graph of genus $g\ge 2$, we prove that $4(g-1)\le M(g)\le 6(g-1)$, and these bounds are sharp in the sense that both are attained for infinitely many values of $g$. Moreover, we show that the values $4(g-1)$ and $6(g-1)$ are the only values taken by the function $M(g)$.
\end{abstract}

\section{Introduction}

In the papers \cite{BN1} and \cite{BN2}, M. Baker and S. Norine study the category of finite graphs with harmonic morphisms as a discrete analogue of the category of Riemann surfaces with holomorphic maps. The cornerstone of the analogy is the Riemann-Roch theorem of \cite{BN1}, which allows for an Abel-Jacobi theory concerning the map from a graph to its Jacobian (see also \cite{BHN} for a different approach). Baker and Norine go on to derive a Riemann-Hurwitz formula in \cite{BN2}, and use it to investigate hyperelliptic graphs and Weierstrass points.  This paper furthers the analogy by establishing graph analogues of the genus bounds for the maximal size of an automorphism group of a compact Riemann surface of genus $g\ge 2$. The classical story is as follows: for each $g\ge 2$, define
$$
N(g):= \max\{|\Aut(X)| \ | \ X \textrm{ is a compact Riemann surface of genus $g$}\}.
$$
Then $8(g+1)\le N(g)\le 84(g-1)$, and these bounds are sharp in the sense that both the upper and lower bound are attained for infinitely many values of $g$. The upper bound $84(g-1)$ was found by Hurwitz in \cite{H}, and the fact that it is sharp was proven by Macbeath in \cite{M}. The lower bound and its sharpness were proven independently by Accola in \cite{A} and Maclachlan in \cite{MacL}, and the techniques of \cite{A} provide the model for many of the arguments in the present paper. Indeed, once armed with the Riemann-Hurwitz formula from \cite{BN2}, the main complication in the graph-theoretic situation is the presence of vertical ramification, which has no analogue in the context of Riemann surfaces.

By a \emph{graph} we will always mean a finite, connected multigraph without loop edges. Thus, multiple edges between two vertices are allowed, but no vertex can have an edge to itself. We denote the edge-set and vertex-set of a graph $G$ by $E(G)$ and $V(G)$ respectively. As in the case of Riemann surfaces, we define the \emph{genus} of a graph $G$ to be the rank of its first Betti homology group, $g(G)=|E(G)|-|V(G)|+1$. This number is the dimension of the cycle-space of $G$, and is more commonly referred to as the cyclomatic number of $G$.

With this terminology, let $M(g)$ be the maximal size of a harmonic group action (see Definition \ref{Hgrpaction}) on any graph of genus $g\ge 2$. This paper proves the following theorem:

\begin{theorem} \label{Main}
For $g\ge 2$ we have $4(g-1)\le M(g)\le 6(g-1)$. In fact, $4(g-1)$ and $6(g-1)$ are the only values taken by the function $M(g)$, and both values are obtained for infinitely many values of $g$.
\end{theorem}

\noindent
Some remarks are in order concerning this result. First, 
Theorem~\ref{Main} does not hold for $g\le 1$. Indeed, in sections~\ref{g=0} and \ref{g=1} we show that there exist arbitrarily large harmonic group actions on graphs of genus 0 and 1. Second, it seems difficult to determine the function $M(g)$ precisely, although perhaps this problem is more tractable than in the case of Riemann surfaces where the analogous function $N(g)$ is still unknown. More generally, it would be interesting to classify the groups and graphs that achieve the upper bound $6(g-1)$. The corresponding question for Riemann surfaces leads to the study of Hurwitz groups and surfaces; for a survey see \cite{Con1} and the update \cite{Con2}.

We begin in section~\ref{HG} by motivating the definition of a harmonic group action, and then characterize these actions as those that are non-degenerate and for which stabilizers of directed edges are trivial (Proposition~\ref{criterion}). In section~\ref{Hbranch} we use the Riemann-Hurwitz formula to investigate harmonic group actions with small ramification, and then in section~\ref{HLowgenus} we classify harmonic actions on graphs of genus 0 and 1. Sections \ref{Hurwitz} and \ref{UpperSharp} establish the sharp upper bound of Theorem~\ref{Main}, while section~\ref{Lower} provides the lower bound. The remaining two sections employ the method of Accola in \cite{A} to prove the sharpness of the lower bound.

\section{Harmonic Group Actions}\label{HG}

A \emph{morphism} of graphs $\phi:G\rightarrow G'$ is a function $\phi:V(G)\cup E(G)\rightarrow V(G')\cup E(G')$ that maps vertices to vertices and such that for all $e\in E(G)$ with endpoints $x,y\in V(G)$, either $\phi(e)$ is an edge of $G'$ with endpoints $\phi(x),\phi(y)$, or $\phi(e)=\phi(x)=\phi(y)$, in which case we say that the edge $e$ is $\phi$-\emph{vertical}. If $x$ is a vertex of $G$, then we define $x(1)$ to be the subgraph of $G$ induced by the edges incident to $x$ (this should be thought of as the smallest neighborhood of $x$ in $G$). Explicitly, the vertices of $x(1)$ are $x$ and the vertices of $G$ adjacent to $x$, while the edges of $x(1)$ are the edges of $G$ incident to $x$. The morphism $\phi$ is \emph{non-degenerate at $x\in V(G)$} provided that $\phi(x(1))\ne \{\phi(x)\}$, and $\phi$ is \emph{non-degenerate} if it is non-degenerate at every vertex of $V(G)$. Thus, $\phi$ is non-degenerate if it does not contract any neighborhood in $G$ to a vertex of $G'$.

In \cite{BN2}, Baker and Norine (building off of Urakawa in \cite{U}) propose the following definition of \emph{harmonic morphism} as the correct graph analogue of a holomorphic map between Riemann surfaces:

\begin{definition}
A morphism $\phi:G\rightarrow G'$ is \emph{harmonic} if for all vertices $x\in V(G)$, the quantity 
$|\phi^{-1}(e')\cap x(1)|$ is independent of the choice of edge $e'\in E(\phi(x)(1))$.
\end{definition}

\noindent
That is, given a point $y\in V(G')$ and a point $x$ in the fiber of $\phi$ over $y$, all edges incident to $y$ have the same number of pre-images incident to $x$, and we denote this common number by $m_\phi(x)$, the \emph{horizontal multiplicity} of $\phi$ at $x$. For each vertex $x\in G$, the \emph{vertical multiplicity} $v_\phi(x)$ is defined to be the number of $\phi$-vertical edges incident to $x$. Finally, the \emph{degree} of $\phi$ is the number of pre-images of any edge in $G'$, and this is shown to be well defined in Lemma 2.4 of \cite{BN2}. 

It is immediate that the composition of harmonic morphisms is harmonic. 
In addition, we have the following result.

\begin{proposition}\label{2outof3} Suppose that $\phi:G_1\rightarrow G_2$ is a nonconstant harmonic morphism, and $\psi:G_2\rightarrow G_3$ is a morphism such that $\rho:=\psi\circ\phi$ is harmonic. Then $\psi$ is harmonic. 
\end{proposition}

\begin{proof}
Given a point $y\in V(G_2)$, let $z=\psi(y)$, and consider any edge $e\in z(1)$. Now pick any $x\in V(G_1)$ such that $\phi(x)=y$ (such an $x$ exists, since by Lemma 2.7 of \cite{BN2}, nonconstant harmonic morphisms are surjective). The cardinality of $\rho^{-1}(e)\cap x(1)$ is independent of our choice of $e$ since $\rho$ is harmonic. Consider the set $\psi^{-1}(e)\cap y(1)$. Then $\phi$ harmonic implies that the cardinality of $\phi^{-1}(f)\cap x(1)$ is independent of the choice of edge $f\in\psi^{-1}(e)\cap y(1)$, and in fact independent of $f\in y(1)$. But
$$
\rho^{-1}(e)\cap x(1)=\bigcup_{f\in\psi^{-1}(e)\cap y(1)}\phi^{-1}(f)\cap x(1),
$$
so it follows that $\#(\psi^{-1}(e)\cap y(1))=\#(\rho^{-1}(e)\cap x(1))/\#(\phi^{-1}(f)\cap x(1))$, independent of $e\in z(1)$. This shows that $\psi$ is harmonic.
\end{proof}

Now suppose that $\Gamma$ is a group of automorphisms of a graph $G$ (note that $\Gamma$ is necessarily finite since $G$ is a finite graph). Following \cite{BN2} we define a quotient graph $G/\Gamma$ together with a graph morphism $\phi_\Gamma:G\rightarrow G/\Gamma$ as follows. The vertices of $G/\Gamma$ are the $\Gamma$-orbits of the vertices of $G$, that is: $V(G/\Gamma):=V(G)/\Gamma$.
The edges of $G/\Gamma$ are the $\Gamma$-orbits of edges of $G$ with endpoints in distinct $\Gamma$-orbits:
$$
E(G/\Gamma)=E(G)/\Gamma - \{\Gamma e \ | \ e \ \textrm{has endpoints $x,y$ and $\Gamma x= \Gamma y$}\}.
$$
Furthermore, $\phi_\Gamma(x)=\Gamma x$ for all $x\in V(G)$.  If $e\in E(G)$ with vertices $x$ and $y$, then $\phi_\Gamma(e)=\Gamma e$ if $\Gamma x\ne\Gamma y$, and $\phi_\Gamma(e)=\phi_\Gamma(x)$ if $\Gamma x=\Gamma y$. That is, edges of $G$ with vertices in the same $\Gamma$-orbit are $\phi_\Gamma$-vertical.

As the following simple example shows, the quotient morphism $\phi_\Gamma$ need not be harmonic. 

\begin{example}\label{barbell}
Consider the $\mathbb{Z}/2\mathbb{Z}$-action on the barbell graph $B$ shown in Figure~\ref{NonHarm}, where the nontrivial element acts by horizontal reflection. The quotient graph is $P_3$, the path of length three, but the quotient morphism $B\rightarrow P_3$ is not harmonic: the center edge of $P_3$ has only one pre-image, while the outer edges each have two.

\begin{figure}[h]
\centering
\begin{tikzpicture}
\tikzstyle{every node}=[circle, draw, fill=black!50, inner sep=0pt, minimum width=3pt]

\node (1) at  (-2,0)  {};
\node (2) at  (-1,0) {} ;        
\node (3) at  (1,0) {};
\node (4) at  (2,0)  {};

\draw[-] (1) to [out=90, in=90] (2);
\draw[-] (1) to [out=270, in=270] (2);
\draw [-] (2) to (3);
\draw [-] (3) to [out=90, in=90] (4);
\draw[-] (3) to [out=270, in=270] (4);

\draw[->] (0,-.25) -- (0,-1.5);

\node (1) at  (-2,-2)  {};
\node (2) at  (-1,-2) {} ;        
\node (3) at  (1,-2) {};
\node (4) at  (2,-2)  {};

\draw[-] (1) to  (2);
\draw [-] (2) to (3);
\draw[-] (3) to  (4);

\end{tikzpicture}
\caption{A non-harmonic quotient morphism with group $\mathbb{Z}/2\mathbb{Z}$.}
\label{NonHarm}
\end{figure}
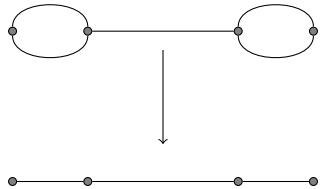

\end{example}

This is contrary to the classical case: if $\Sigma$ is a finite group of holomorphic automorphisms of a Riemann surface $X$, then the quotient space $X/\Sigma$ has a complex structure such that the quotient map $X\rightarrow X/\Sigma$ is holomorphic (\cite{Mir} Theorem 3.4). To remedy this flaw in the graph-theoretic context, we introduce the following definition.

\begin{definition}\label{Hgrpaction}
Suppose that $\Gamma<\Aut(G)$ is a group of automorphisms of a graph $G$. Then $\Gamma$ \emph{acts harmonically} on $G$ if for all subgroups $\Delta<\Gamma$, the quotient morphism $\phi_\Delta:G\rightarrow G/\Delta$ is harmonic and non-degenerate. 
\end{definition}

\noindent
Proposition \ref{2outof3} has the following consequence: if $\Gamma$ is a group acting harmonically on $G$, and $\Delta_1<\Delta_2<\Gamma$ are subgroups, then the induced map $G/\Delta_1\rightarrow G/\Delta_2$ is harmonic and non-degenerate.

The requirement that all subgroups induce harmonic quotient morphisms is a very stringent one. As a simple illustration, consider the $\mathbb{Z}/2\mathbb{Z}\times\mathbb{Z}/2\mathbb{Z}$-action on the barbell graph $B$ in Figure~\ref{NonHarm}, where the first copy of $\mathbb{Z}/2\mathbb{Z}$ acts by reflection across the vertical bisector of the middle edge, while the second copy acts via a counterclockwise rotation through $\pi$ radians. Both of these involutions act harmonically, with quotients of genus 1. Moreover, the full Klein-four action induces a harmonic quotient morphism to $P_1$, the path of length one. Nevertheless, this action is not harmonic, because the composition of the two involutions described above is the horizontal reflection described in Example~\ref{barbell}.

The next proposition provides a criterion for a group action to be harmonic.

\begin{proposition}\label{criterion}
Suppose $\Gamma<\Aut(G)$ is a group of automorphisms of a graph $G$. Then $\Gamma$ acts harmonically if and only if for every vertex $x\in V(G)$, the stabilizer subgroup $\Gamma_x$ acts freely on $E(x(1))$ and $V(x(1))\not\subset \Gamma x$. 
\end{proposition}

\begin{proof}
It is clear that all quotient maps by subgroups of $\Gamma$ are non-degenerate if and only if $V(x(1))\not\subset \Gamma x$ for all $x\in V(G)$, so assume that this condition is satisfied.

First suppose that the stabilizer subgroups $\Gamma_x$ act freely on the edgesets $E(x(1))$ for all $x\in V(G)$. Let $\Delta<\Gamma$ be any subgroup, and consider the quotient map $\phi_\Delta:G\rightarrow G/\Delta$. Furthermore, let $e'\in y(1)$ where $y$ is a vertex of $G/\Delta$, and choose any $x\in G$ such that $\phi_\Delta(x)=y$. Then choose an edge $e\in x(1)$ such that $\phi_\Delta(e)=e'$. Now $\phi_\Delta^{-1}(e')\cap x(1)=\Delta e\cap x(1)=\Delta_x e=(\Gamma_x \cap \Delta)e$. But since $\Gamma_x$ acts freely on $E(x(1))$, we see that the cardinality of $\phi_\Delta^{-1}(e')\cap x(1)$ is equal to the order of $\Gamma_x\cap \Delta=\Delta_x$, independently of $e$. Thus $\Gamma$ acts harmonically on $G$.

Now suppose that $\Gamma$ acts harmonically on $G$, and consider a stabilizer $\Gamma_x$. Assume that $\gamma\in \Gamma_x$ fixes the edge $e\in x(1)$. Then consider the non-degenerate harmonic quotient map $\phi_{\gamma}:G\rightarrow G/\left<\gamma\right>$. Setting $\phi_\gamma(e)=e'$, we see that $\deg(\phi_\gamma)=\#(\phi_\gamma^{-1}(e'))=\#(\left<\gamma\right>e)=1$, so that $\phi_\gamma$ is a degree one harmonic morphism. It follows that every edge of $G$ is either fixed by $\gamma$, or else the ends of the edge are contained in the same $\left<\gamma\right>$-orbit. But then every vertex $w\in G$ is either fixed by $\gamma$, or else $V(w(1))\subset \left<\gamma\right>w$, which contradicts non-degeneracy. Thus $\gamma$ fixes all vertices of $G$, and hence is the identity automorphism. This shows that the stabilizer $\Gamma_x$ acts freely on $E(x(1))$.
\end{proof}

Proposition~\ref{criterion} says that  $\Gamma$ acts harmonically on $G$ if and only if the stabilizers of directed edges are trivial, and $V(x(1))\not\subset \Gamma x$ for all $x\in V(G)$. Thus, if $\Gamma$ acts harmonically on $G$, then the quotient map $G\rightarrow G/\Gamma$ has degree $|\Gamma|$ as expected. Moreover, at every $x\in G$, the horizontal multiplicity of $\phi_\Gamma$ is given by $m(x)=|\Gamma_x|$, and $\Gamma_x$ permutes the $\phi_\Gamma$-vertical edges in $x(1)$. Since the action is free, we see that $|\Gamma_x|$ divides the vertical multiplicity $v(x)$, and we define $w(x)$ so that $w(x)|\Gamma_x|=v(x)$. The Riemann-Hurwitz formula (\cite{BN2} Theorem 2.14) now says that

\begin{eqnarray*}
2g(G)-2&=&|\Gamma|(2g(G/\Gamma)-2)+\sum_{x\in V(G)}2(m(x)-1)+v(x)\\
&=&|\Gamma|(2g(G/\Gamma)-2)+\sum_{x\in V(G)}2(|\Gamma_x|-1)+v(x)\\
&=&|\Gamma|(2g(G/\Gamma)-2)+\sum_{y=\phi_\Gamma(x)\in V(G/\Gamma)}\frac{|\Gamma|}{|\Gamma_x|}[2(|\Gamma_x|-1)+v(x)]\\
&=&|\Gamma|(2g(G/\Gamma)-2+\sum_{y=\phi_\Gamma(x)\in V(G/\Gamma)}[2(1-\frac{1}{|\Gamma_x|})+w(x)])\\
&=&|\Gamma|(2g(G/\Gamma)-2+R)
\end{eqnarray*}
where $R:=\sum_{y\in V(G/\Gamma)}[2(1-\frac{1}{r_y})+w_y]$, with $r_y:=|\Gamma_x|$ and $w_y:=w(x)$ for any $x\in\phi_\Gamma^{-1}(y)$.

We end this section with two examples to illustrate the meaning of the quantities $r$ and $w$ in the Riemann-Hurwitz formula above.

\begin{example}
Figure~\ref{r1w1} shows a harmonic Klein-four action on a graph of genus 3, yielding a quotient of genus 1. One involution acts by counterclockwise rotation through $\pi$ radians, while another involution exchanges the top and bottom squares, flipping the vertical edges. The branch locus consists of a single point $y$, with $r_y=w_y=1$.

\begin{figure}[h]
\centering
\begin{tikzpicture}
\tikzstyle{every node}=[circle, draw, fill=black!50, inner sep=0pt, minimum width=3pt]

\node (1) at  (-1,1)  {};
\node (2) at  (0,1.5) {} ;        
\node (3) at  (1,1) {};
\node (4) at  (0,.5)  {};

\node (5) at  (-1,.5)  {};
\node (6) at  (0,1) {} ;        
\node (7) at  (1,.5) {};
\node (8) at  (0,0)  {};

\draw[-] (1) to (2);
\draw[-] (2) to (3);
\draw [-] (3) to (4);
\draw [-] (4) to (1);

\draw[-] (5) to (6);
\draw[-] (6) to (7);
\draw[-] (7) to (8);
\draw[-] (8) to (5);

\draw[-] (1) to (5);
\draw[-] (3) to (7);

\draw[->] (0,-.25) -- (0,-1.5);

\node (1) at  (0,-2) {};
\node (2) at  (1,-2) {};  
\coordinate [label=right : $y$] (y) at (1.1,-2);
  
\draw[-] (1) to [out=60, in=120] (2);
\draw[-] (2) to [out=240, in=300] (1);

\end{tikzpicture}
\caption{A harmonic Klein-four cover with only vertical ramification.}
\label{r1w1}
\end{figure}
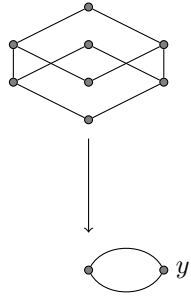

\end{example}

\begin{example}
Figure~\ref{r2w1} shows a harmonic Klein-four action on a graph of genus 5, yielding a quotient of genus 1. One involution acts by reflection across the horizontal bisector of the vertical edges, while another involution fixes the degree 6 vertices on the right, exchanging the horizontal petals and switching the two vertical edges. The branch locus consists of a single point $y$, with $r_y=2$ and $w_y=1$.

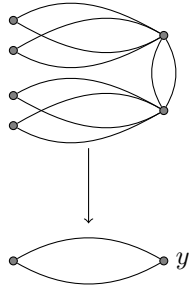
\begin{figure}[h]
\centering
\begin{tikzpicture}
\tikzstyle{every node}=[circle, draw, fill=black!50, inner sep=0pt, minimum width=3pt]

\node (1a) at  (-1,1.2)  {};
\node (1b) at  (-1,.8)  {};
\node (2) at  (1,1) {} ;        
\node (3a) at  (-1,.2) {};
\node (3b) at  (-1,-.2) {};
\node (4) at  (1,0)  {};

\draw[-] (1a) to [out=30, in=150] (2);
\draw[-] (1a) to [out=330, in=210] (2);
\draw[-] (1b) to [out=30, in=150] (2);
\draw[-] (1b) to [out=330, in=210] (2);
\draw[-] (3a) to [out=30, in=150] (4);
\draw[-] (3a) to [out=330, in=210] (4);
\draw[-] (3b) to [out=30, in=150] (4);
\draw[-] (3b) to [out=330, in=210] (4);
\draw[-] (2) to [out=240, in=120] (4);
\draw[-] (2) to [out=300, in=60] (4);

\draw[->] (0,-.5) -- (0,-1.5);

\node (1) at  (-1,-2) {};
\node (2) at  (1,-2) {};  
\coordinate [label=right : $y$] (y) at (1.1,-2);
  
\draw[-] (1) to [out=30, in=150] (2);
\draw[-] (1) to [out=330, in=210] (2);

\end{tikzpicture}
\caption{A harmonic Klein-four cover with horizontal and vertical ramification.}
\label{r2w1}
\end{figure}

\end{example}

\section{Branch Loci of Harmonic Quotient Maps}\label{Hbranch}

Let $\Gamma$ be a group acting harmonically on a finite graph $G$, so that we have a non-degenerate harmonic quotient map $\phi_\Gamma:G\rightarrow G/\Gamma=G'$. In this section we describe how the ramification number $R=\sum_{y\in G'}2(1-\frac{1}{r_y})+w_y$ from the Riemann-Hurwitz formula determines the possible branch loci for $\phi_\Gamma$ in several cases.

\begin{proposition}\label{R<2} 
If $0<R<2$, then there is a single branch point $y\in G'$, and either $r_y\ge 2$ and $w_y=0$, or $r_y=w_y=1$.
\end{proposition}

\begin{proof}
First note that each term $2(1-\frac{1}{r_y})+w_y\ge 1$, with equality only if $r_y=w_y=1$ or $r_y=2, w_y=0$. This shows immediately that there is only one branch point if $0<R<2$. Moreover, the cases listed are the only possible yielding $0<R<2$. The case $r_y=w_y=1$ is illustrated in Figure~\ref{r1w1}.
\end{proof}

\begin{proposition}\label{R=2}
If $R=2$, then there are several possibilities:
\begin{itemize}
\item[(i)] There is a single branch point $y\in G'$, and $r_y=2, w_y=1$ or $r_y=1, w_y=2$.
\item[(ii)] There are two branch points $y_1$ and $y_2,$ and we have the following possibilities for the ramification vector $(r_1,r_2; w_1,w_2)$ (up to a re-ordering of the branch points):
\begin{itemize}
\item[(a)] (2,2; 0,0)
\item[(b)] (1,1; 1,1)
\item[(c)] (1,2; 1,0).
\end{itemize}
\end{itemize}
\end{proposition}

\begin{proof}
If there is a single branch point, then we must have $\frac{2}{r_y}=w_y$ or $2=r_yw_y$. The two solutions are as stated. The case $r_y=2$ and $w_y=1$ is illustrated in Figure~\ref{r2w1}.

Now suppose there are two branch points. Then we must have $2+w_1+w_2=\frac{2}{r_1}+\frac{2}{r_2}$. The three vectors listed are clearly solutions, and a little thought shows that they are the only solutions. 

More than two branch points cannot occur, since each contributes at least 1 to $R$.
\end{proof}

\begin{proposition}\label{R>2}
If $R>2$, then in fact $R\ge\frac{7}{3}$, and this lower bound occurs in exactly three ways:
\begin{itemize}
\item[(i)] a single branch point with $r=3, w=1$;
\item[(ii)] two branch points with vector $(r_1,r_2; w_1,w_2)=(3,2; 0,0)$;
\item[(iii)] two branch points with vector $(r_1,r_2; w_1,w_2)=(3,1; 0,1)$.
\end{itemize}
\end{proposition}

\begin{proof}
First note that if there are three or more branch points, then $R\ge 3$. For a single branch point we have
$2<R=2-\frac{2}{r}+w$ if and only if $2< rw$, and a little thought shows that the minimum occurs as stated. Moreover, since the $w$'s only contribute integer values, the only way to achieve $R=\frac{7}{3}$ is to get a contribution of $\frac{4}{3}$ from one of the $r$ terms. This can happen in the two ways stated. \end{proof}

\section{Harmonic Group Actions on Graphs of Low Genus}\label{HLowgenus}

\subsection{Genus zero}\label{g=0}

Suppose that $T$ is a tree (i.e. a finite graph of genus zero), and that $\Gamma$ acts harmonically on $T$. Then Riemann-Hurwitz says that $-2=|\Gamma|(-2+R)$. If $\Gamma$ is nontrivial, then we see that $R\ne 0$, and in fact $0<R<2$ with $|\Gamma|=\frac{2}{2-R}$. From Proposition~\ref{R<2} we see that $T\rightarrow T/\Gamma$ has a single branch point $y$ with either $r=w=1$ (which yields $R=1$) or $r\ge 2, w=0$ (which yields $R=2(1-\frac{1}{r})$).  In the first case we get $\Gamma=\mathbb{Z}/2\mathbb{Z}$, and the tree $T$ can be obtained by connecting two copies of a rooted tree $T_0$ with a single edge between the roots. Then $\Gamma$ acts by interchanging the two rooted trees, so that the ramification locus consists of the two roots (each with horizontal multiplicity 1), and the connecting edge is $\phi_\Gamma$-vertical.

In the second case we get $|\Gamma|=r$, which implies that the ramification locus consists of a single point $x$ (because the order of the stabilizer of $x$ is $r$). Then $T$ is obtained as follows: let $T_0$ be a rooted tree, and then let $T$ be the tree obtained by joining $r$ copies of $T_0$ at their roots, one for each element of $\Gamma$. Then let $\Gamma$ act on $T$ by permuting the copies of $T_0$ according to the left-regular representation. In particular, we see that $T$ must have a central vertex, with isomorphic branches. This analysis shows that unlike the case of Riemann surfaces (\cite{Mir} Prop. 3.1), stabilizers of harmonic group actions can be non-cyclic ($\Gamma$ was arbitrary above).

\subsection{Genus one}\label{g=1}

Now suppose that $G$ has genus 1, and $\Gamma$ acts harmonically on $G$. If $g(G/\Gamma)=1$, then Riemann-Hurwitz says that $0=|\Gamma|R$, so that $R=0$ and the quotient map is unramified. It follows that $\Gamma$ must be a subgroup of the rotations of the unique cycle of $G$. In particular, $\Gamma$ is cyclic of order a strict divisor of the number of vertices in the unique cycle of $G$.

Now we consider the case that $g(G/\Gamma)=0$. Then we find that $R=2$, and Proposition \ref{R=2} tells us that there are several possibilities for the ramification locus. The first case is that there is a single branch point $y$, and every point above $y$ has horizontal multiplicity 2 and 2 vertical edges. Hence, the stabilizers are all cyclic of order 2. Here is the only example: the obvious $D_n$-action on an $n$-cycle decorated with two copies of a tree $T$ at each vertex. The other ramification possibilities arise by restricting this $D_n$-action to various subgroups as described below.

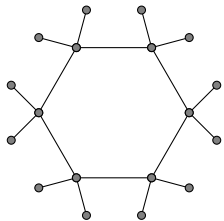
\begin{figure}[h]
\centering
\begin{tikzpicture}
\tikzstyle{every node}=[circle, draw, fill=black!50, inner sep=0pt, minimum width=3pt]

\foreach \x in {0,60,120,180,240,300}
\draw[rotate around={\x:(0,0)}] (-.5,.866) node {} to (.5,.866) node {}
(-1,1) node {} to (-.5,.866)
(1,1) node {} to (.5,.866);

\end{tikzpicture}
\caption{A decorated 6-cycle supporting a harmonic $D_6$-action.}
\label{genus1action}
\end{figure}

In the second case there is a single branch point $y$, and every point above $y$ has horizontal multiplicity 1 and 2 vertical edges (in this case all ramification is vertical). This corresponds to the action of the cyclic subgroup of degree $n$ contained in $D_n$. The remaining cases all have two ramification points, and differ according to their ramification vector $(r_1,r_2; w_1,w_2)$. The case $(2,2; 0,0)$ requires $n$ to be even, and corresponds to a reflection across the line joining two antipodal vertices in the $n$-cycle. The case $(1,2; 1,0)$ requires $n$ to be odd, and corresponds to a reflection across the line joining a vertex to the midpoint of the opposite edge. Finally, the case $(1,1; 1,1)$ occurs by reflecting an $n$-cycle for $n$-even across the line joining the midpoints of two opposite edges.

\section{The Hurwitz Genus Bound}\label{Hurwitz}

Now suppose that $g(G)\ge 2$, and $\Gamma$ acts harmonically on $G$. First assume that $g(G/\Gamma)\ge 1$. If $R=0$, then the Riemann-Hurwitz formula implies that $g(G/\Gamma)\ge 2$, so
$$
|\Gamma|=\frac{g(G)-1}{g(G/\Gamma)-1}\le g(G)-1.
$$
If $R\ne 0$, then $R\ge 1$, so
$$
|\Gamma|=\frac{2g(G)-2}{2g(G/\Gamma)-2+R}\le 2g(G)-2.
$$
Consequently, if $|\Gamma|>2(g(G)-1)$, then we must have $g(G/\Gamma)=0$. 

We now assume that $g(G/\Gamma)=0$. Then $R>2$, and so by Proposition \ref{R>2} we actually have $R\ge\frac{7}{3}$. But then
$$
|\Gamma|=\frac{2g(G)-2}{R-2}\le 6(g(G)-1).
$$
This is the analogue of the classical Hurwitz genus bound \cite{H} for the size of holomorphic group actions on compact Riemann surfaces of genus $g\ge 2$. 

The following example shows that the upper bound $6(g-1)$ is attained for $g=2$: let $T$ be the rooted tree consisting of three edges emanating from a root vertex. Then join two copies of $T$ via 3 edges between the roots to obtain a graph $G$ of genus 2 (see Figure~\ref{Hurwitzg2}). The group $\mathbb{Z}/3\mathbb{Z}\times \mathbb{Z}/2\mathbb{Z}=\left<\sigma,\tau\right>$ acts harmonically on $G$, achieving the Hurwitz genus bound. Here $\sigma$ fixes both degree-6 vertices while permuting all the edges, and $\tau$ flips the vertices, interchanging the two copies of $T$ and making the connecting edges $\phi_\tau$-vertical.

\begin{figure}[h]
\centering
\begin{tikzpicture}
\tikzstyle{every node}=[circle, draw, fill=black!50, inner sep=0pt, minimum width=3pt]

\node (L1) at  (-2,1)  {};
\node (L2) at  (-2,0) {} ;        
\node (L3) at  (-2,-1) {};
\node (CL) at  (-1,0)  {};
\node (CR) at  (1,0) {};
\node (R1) at  (2,1) {};
\node (R2) at  (2,0) {};
\node (R3) at  (2,-1) {};

\draw[-] (CL) to [out=45, in=135] (CR);
\draw[-] (CL) to [out=-45, in=-135] (CR);
\draw [-] (CL) -- (CR);
\draw [-] (CR) -- (R1);
\draw [-] (CR) -- (R2);
\draw [-] (CR) -- (R3);
\draw [-] (CL) -- (L1);
\draw [-] (CL) -- (L2);
\draw [-] (CL) -- (L3);
\end{tikzpicture}
\caption{A genus 2 graph supporting a harmonic $\mathbb{Z}/3\mathbb{Z}\times\mathbb{Z}/2\mathbb{Z}$-action, thereby attaining the Hurwitz bound.}
\label{Hurwitzg2}
\end{figure}
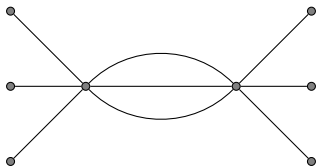

\section{Sharpness of the Upper Bound}\label{UpperSharp}

Adapting an idea from Macbeath \cite{M}, we prove the following result:

\begin{proposition}\label{UpperExist}
Let $g=m^2+1$ for $m\ge 1$. Then $M(g)=6(g-1)$. That is, there exists a graph $G(m)$ of genus $g=m^2+1$ and a group $\Gamma(m)$ of order $6(g-1)=6m^2$ acting harmonically on $G(m)$.
\end{proposition}

\begin{proof}
Let $G=G(1)$ be the graph of genus 2 achieving the Hurwitz bound in Figure~\ref{Hurwitzg2}. Then the fundamental group $\pi_1$ of $G$ is the free group on two generators, corresponding to the two loops of $G$. As described at the end of section~\ref{Hurwitz}, $\Gamma=\mathbb{Z}/3\mathbb{Z}\times\mathbb{Z}/2\mathbb{Z}=\left<\sigma,\tau\right>$ acts harmonically on $G$, and by Theorem 2 of \cite{M} this action lifts to the universal cover $T$ of $G$, normalizing $\pi_1$ viewed as the group of deck transformations. Thus, we have a group $\widetilde{\Gamma}$ of automorphisms of $T$, described by the following group extension:
$$
1\rightarrow\pi_1\rightarrow\widetilde{\Gamma}\rightarrow \Gamma\rightarrow 1.
$$
Taking the pushout via the characteristic quotient $\pi_1\rightarrow\pi_1/[\pi_1,\pi_1]\pi_1^m\cong (\mathbb{Z}/m\mathbb{Z})^2$, we obtain an exact sequence defining a group $\Gamma(m)$ of order $6m^2$:
$$
1\rightarrow (\mathbb{Z}/m\mathbb{Z})^2\rightarrow \Gamma(m)\rightarrow \Gamma\rightarrow 1.
$$
Letting $G(m)$ denote the $(\mathbb{Z}/m\mathbb{Z})^2$ unramified covering of $G$ corresponding to the subgroup $[\pi_1,\pi_1]\pi_1^m$, we see that $\Gamma(m)$ acts on $G(m)$, and an application of Proposition~\ref{criterion} shows that this action is harmonic. Indeed, the vertices of $G(m)$ with non-trivial stabilizers comprise the fibers over the degree-6 vertices of $G$. For each such vertex $x\in G(m)$, the stabilizer subgroup $\Gamma(m)_x$ has order 3, and it acts freely on the six edges incident to $x$. Thus, no directed edge is fixed by $\Gamma(m)$, so the action is harmonic. Moreover, by Riemann-Hurwitz $g(G(m))=m^2(g(G)-1)+1=m^2+1$.
\end{proof}

\section{A Lower Bound}\label{Lower}

Let $S$ denote the square graph consisting of four vertices and four edges. 
For $g\ge 3$, let $G_{g}$ be the graph of genus $g$ obtained by joining $g-1$ copies of $S$ in a cycle, leaving a pair of opposite vertices in each square unattached. The group $\mathbb{Z}/2\mathbb{Z}\times D_{g-1}$ acts harmonically on the resulting graph $G_g$. Here $D_{g-1}$ acts in the natural way on the cycle of squares, while $\mathbb{Z}/2\mathbb{Z}$ interchanges the inner and outer edges of the squares.

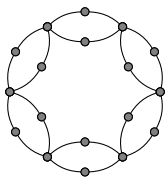
\begin{figure}[h]
\centering
\begin{tikzpicture}

\tikzstyle{every node}=[circle, draw, fill=black!50, inner sep=0pt, minimum width=3pt]

\foreach \x in {0,60,120,180,240,300}
\draw[rotate around={\x:(0,0)}] (-.5,.866) node {} to [out=45,in=180] (0,1.066) node {}
(0,1.066) node {} to [out=0,in=135] (.5,.866) node {}
(.5,.866) node {} to [out=-135,in=0] (0,.666) node {}
(0,.666) node {} to  [out=180,in=-45] (-.5,.866) node {};

\end{tikzpicture}

\caption{The graph $G_7$}
\label{LowerBoundGraph}
\end{figure}

Hence, for each $g\ge 3$, there exists a graph $G_g$ of genus $g$ supporting a harmonic group action by a group of order $4(g-1)$. If we define $M(g)$ to be the maximal size of a group acting harmonically on any graph of genus $g\ge 2$, then we have established the bounds announced in Theorem~\ref{Main}:
$$
4(g-1)\le M(g)\le 6(g-1),
$$
and by Proposition~\ref{UpperExist}, the upper bound is sharp. To complete the proof of Theorem~\ref{Main}, it remains to show that the lower bound is sharp, and that $4(g-1)$ and $6(g-1)$ are actually the only values taken by the function $M(g)$.

\section{Prime Factors of Harmonic Automorphism Group Orders}\label{Primes}

Suppose that $g=g(G)\ge 2$, and $\Gamma$ acts harmonically on $G$. We wish to determine the possible prime factors of $|\Gamma|$. As in Accola \cite{A}, the next two lemmas are simple consequences of the Riemann-Hurwitz formula. So suppose that $p \ | \ |\Gamma|$, for $p$ prime. Then by Cauchy's Theorem, $\Gamma$ has a cyclic subgroup, $C$, of order $p$. We consider the $p$-cyclic harmonic cover $G\rightarrow G/C$. Note that if $y\in G/C$ is a branch point with horizontal ramification in the fiber, then the fiber consists of a single point (since $p$ prime), so that there is no vertical ramification over $y$. Hence, the Riemann-Hurwitz formula says that
$$
2g-2=p(2g_0-2)+2s(p-1)+p\sum_{j=1}^tw_j
$$
where 
\begin{itemize}
\item $g_0:=g(G/C)$;
\item $s:=$ number of branch points with horizontal ramification in the fiber;
\item $t:=$ number of branch points with only vertical ramification in the fiber;
\item $w_j\ge 1$ is the vertical ramification index above the $j$th vertically ramified branch point. 
\end{itemize}
First suppose that $g_0\ge 2$. Then we have $2g-2\ge 2p$, or $p\le g-1$.
If $g_0=1$, then the Riemann-Hurwitz formula becomes
$$
2g-2=2s(p-1)+p\sum_{j=1}^t w_j.
$$
so that if $s\ge 1$ we have $2g-2\ge 2(p-1)$, or $p\le g$.
On the other hand if $s=0$, then we have $2g-2=p\sum_{j=1}^t w_j$, so that $p \ | \ 2(g-1)$, and in particular
$p\le \max\{2,g-1\}$.

Now suppose that $g_0=0$. Then we have
$$
2g-2=-2p+2s(p-1)+p\sum_{j=1}^tw_j.
$$
If $s\ge 2$ then we get $2g-2\ge 2p-4$, or $p\le g+1$.
If $s=1$, then we get $2g=p\sum_{j=1}^tw_j$, which implies that $p \ | \ 2g$. In particular,
$p\le g$. Finally, if $s=0$, then $2g-2=-2p+p\sum_{j=1}^tw_j$, which implies that $p \ | \ 2(g-1)$, and in particular
$p\le \max\{2,g-1\}$.
Putting all of this together, we see that $p\le g+1$ in all cases, and we have proven the following lemma:

\begin{lemma}\label{factorial}
Suppose that $\Gamma$ acts harmonically on a graph $G$ of genus $g\ge 2$. If $p$ is a prime factor of $|\Gamma|$, then $p$ divides $(g+1)!$.
\end{lemma}

We now prove the analogue of Lemma 10 from \cite{A}:

\begin{lemma}\label{unram}
Let $g_1>1$ and set $g-1=p(g_1-1)$ for a prime number $p$. Suppose that $C$ is a cyclic group of order $p$ acting harmonically on a graph $G$ of genus $g$. If $p>g_1+1$, then the $p$-cyclic harmonic cover $G\rightarrow G/C$ is horizontally unramified.
\end{lemma}

\begin{proof}
Riemann Hurwitz says that (using the notation above):
$$
2g-2=p(2g_0-2)+2s(p-1)+p\sum_{j=1}^tw_j.
$$
Using the fact that $g-1=p(g_1-1)$ gives
$$
2p(g_1-1)=p(2g_0-2)+2s(p-1)+p\sum_{j=1}^tw_j.
$$
Subtraction now yields
$$
2s(p-1)=2p(g_1-g_0)-p\sum_{j=1}^tw_j,
$$
which implies that $p$ divides $s$ since $p>g_1+1>2$.
If $g_1=g_0$, then we see that the map must be unramified. So suppose that $g_1>g_0$ (so the map \emph{is} ramified). If there were horizontal ramification, then $s>0$, and it would follow that 
$$
p-1 = \frac{p}{s}(g_1-g_0)-\frac{p}{2s}\sum_{j=1}^tw_j\le g_1,
$$
so $p\le g_1+1$, contradicting our hypothesis on $p$. Thus, the map is horizontally unramified as claimed.
\end{proof}

\begin{lemma}\label{lowgenus}
Suppose that $\Gamma$ acts harmonically on a graph $G$ of genus $g>2$, and that $C<\Gamma$ is a normal $p$-cyclic subgroup for a prime $p<g$. Furthermore, suppose that the quotient map $G\rightarrow G/C$ is horizontally unramified. If $g(G/C)\le 1$, then $|\Gamma|\le 4(g-1)$.
\end{lemma}

\begin{proof}
Set $k=|\Gamma/C|$, so that $|\Gamma|=pk$. Note that if $k\le 4$, then $|\Gamma|=kp\le 4(g-1)$ by assumption. So assume that $k>4$. First consider the case $g(G/C)=1$. Since $g> 2$, the $p$-cyclic cover $G\rightarrow G/C$ must be ramified, and by assumption all ramification is vertical. So Riemann-Hurwitz says that 
$2g-2=p\sum_{j=1}^t w_j\ge pt$. But $\Gamma/C$ acts harmonically on $G/C$, and it must permute the branch locus of the cover $G\rightarrow G/C$. From section~\ref{g=1} we see that $\Gamma/C$ is isomorphic to a subgroup of a dihedral group, and that the stabilizers of its action on $G/C$ have order at most 2. So the branch locus of $G\rightarrow G/C$ consists of a disjoint union of $\Gamma/C$-orbits, each of size at least $\frac{k}{2}$. It follows that $t\ge\frac{k}{2}$, so that $4g-4=2(2g-2)\ge 2pt\ge 2p\frac{k}{2}=pk=|\Gamma|$ as claimed.

Now suppose that $g(G/C)=0$. Then $\Gamma/C$ acts harmonically on the tree $G/C$, and from section~\ref{g=0} we see that if $G/C$ does not have a central vertex with isomorphic branches, then $k=2$, contrary to our assumption $k>4$. So assume that $G/C$ does have a central vertex with isomorphic branches, and consider the vertically ramified $p$-cyclic cover $G\rightarrow G/C$. First suppose that only the center point of $G/C$ is ramified in the cover. Then the fiber over the center point forms a single $\Gamma$-orbit, so each of its $p$ vertices has a stabilizer of order $k$ under the action of $\Gamma$. But this means that the map $G\rightarrow G/\Gamma$ is branched at a single point, with $r=k$ and $w\ge 1$. Riemann-Hurwitz yields $2g-2=pk(-2+2(1-\frac{1}{k})+w)=-2p+pkw\ge pk-2p$. Multiplying by 2  and using $k>4$ then gives $4g-4\ge 2pk-4p>2pk-pk=pk=|\Gamma|.$

If the branch locus of $G\rightarrow G/C$ contains points other than the central point, then the set of non-central branch points forms a disjoint union of $\Gamma/C$-orbits, each of size $k$. In particular, we have $4< k\le t$, where $t$ is the number of branch points. Riemann-Hurwitz says that $2g-2=-2p+p\sum_{j=1}^tw_j\ge -2p+pt=p(t-2)$. It follows that $|\Gamma|=pk\le pt=2pt-pt< 2pt-4p=2p(t-2)\le 4g-4$ as required.

\end{proof}

Finally, we prove an analogue of Theorem 1 from \cite{A}. The new wrinkle in the graph-theoretic scenario is the possibility of vertical ramification.

\begin{proposition}\label{psmall}
Let $g_1>1$ and $p$ be a prime satisfying the following two conditions:
\begin{itemize}
\item[(i)] $p> 6(g_1-1)$
\item[(ii)] $(\frac{p-1}{2}, (g_1+1)!)=1$.
\end{itemize}
Set $g-1=p(g_1-1)$ and suppose that $G$ is a graph of genus $g$ admitting a group $\Gamma$ of automorphisms acting harmonically. If $p$ divides $|\Gamma|$, then $|\Gamma|\le4(g-1)$.
\end{proposition}

\begin{proof}
Suppose that $p \ | \ |\Gamma|$. By the Hurwitz genus bound we have $|\Gamma|\le 6(g-1)=6p(g_1-1)$, so that
$$
k:=\frac{|\Gamma|}{p}\le 6(g_1-1)<p.
$$
Thus, we have $|\Gamma|=pk$ with $p>k$. By the Sylow Theorems, $G$ has a unique (normal) cyclic $p$-Sylow subgroup $C$. Moreover, $p>6(g_1-1)>g_1+1$, so by Lemma~\ref{unram}, the map $G\rightarrow G/C$ is horizontally unramified. Riemann-Hurwitz then says that $2p(g_1-1)=2g-2=p(2g(G/C)-2)+p\sum_{j=1}^tw_j\ge p(2g(G/C)-2)$, so that $g(G/C)\le g_1$. If $g(G/C)\le 1$, then Lemma~\ref{lowgenus} says that $|\Gamma|\le 4(g-1)$, so we may assume that $g(G/C)>1$.

Now the quotient group $K:=\Gamma/C$ of order $k$ acts harmonically on $G/C$, so Lemma~\ref{factorial} says that all prime factors of $k$ must divide $(g(G/C)+1)!$, a fortiori $(g_1+1)!$. But by assumption $(ii)$ on the prime $p$, we conclude that $\frac{p-1}{2}$ is relatively prime to $k$, which implies that $(p-1,k)=1,2$. 

Consider the map $\Gamma\rightarrow\Aut(C)\cong\mathbb{Z}/(p-1)\mathbb{Z}$, defined by sending an element $\gamma\in \Gamma$ to the automorphism of $C$ given by conjugation by $\gamma$. If $H$ is the kernel of this map, then $H$ is a normal subgroup of $\Gamma$ containing $C$ in its center. Now $|\Gamma/H|$ divides both $|\Gamma|=pk$, and $|\Aut(C)|=p-1$, so that $|\Gamma/H|$ divides $(p-1,k)=1,2$. This shows that $H$ has index 1 or 2 in $\Gamma$. Hence, we will be finished if we can show that $|H|\le 2g-2$. 

Now $H$ is a central extension of $H/C$ by $C$. But these groups have relatively prime orders, so the extension must be split. Hence, $H=C\times N$, where $N$ is a normal complement to $C$ in $H$. I claim that the $p$-cyclic harmonic cover $G/N\rightarrow G/H$ is horizontally unramified. Indeed, if $y\in G/N$ is a horizontally ramified point, then it must be totally ramified, so that the image of a generator $c\in C$ fixes $y$. Let $x\in G$ be a preimage of the point $y$ under the map $G\rightarrow G/N$. Then $c$ must act on the orbit $Nx\subset V(G)$, which has cardinality prime to $p$ (it is a factor of $k$). Hence, $c$ must act trivially on $Nx$, which means that the action of $c$ on $G$ has fixed points. But then $G\rightarrow G/C$ is horizontally ramified, which is a contradiction.  So the map $G/N\rightarrow G/H$ can have at most vertical ramification. 

If $G/N\rightarrow G/H$ is unramified, then $2g(G/N)-2=p(2g(G/H)-2)$, which implies that $g(G/H)\ge 1$. From Section~\ref{Hurwitz} we see that $|H|\le 2g-2$. Now suppose that $G/N\rightarrow G/H$ is vertically ramified. If $g(G/H)>0$, then as before we have $|H|\le 2g-2$. So suppose on the contrary that $g(G/H)=0$. Then Riemann-Hurwitz says that
$$
2(g(G/N)-1)=-2p+p\sum_{j=1}^tw_j.
$$
In particular, we see that $p$ divides $g(G/N)-1$. First suppose that $g(G/N)>1$. Then we must have $p\le g(G/N)-1$. Now since $|N|=k$ or $\frac{k}{2}$, Riemann-Hurwitz applied to the map $G\rightarrow G/N$ says that 
$$
2(g-1)=2|N|(g(G/N)-1)+|N|R\ge k(g(G/N)-1).
$$
Multiplying by $p$ and dividing by $g(G/N)-1$ yields
$$
|\Gamma|=pk\le\frac{2p(g-1)}{g(G/N)-1}\le 2g-2.
$$

Now consider the case $g(G/H)=0$ and $g(G/N)=1$. Then the vertically ramified map $G/N\rightarrow G/H$ must have $\sum_{j=1}^tw_j=2$. If there are two branch points, then reviewing the options for harmonic actions on genus 1 graphs (Section~\ref{g=1}), we see that the only possibility is $p=2$, which is excluded by our assumption that $p>6(g_1-1)$. Hence, there is a single branch point, and we see that $G/N$ is a $p$-cycle decorated with a tree $T$ at each vertex. 

Now consider the map $G\rightarrow G/N$. Then Riemann-Hurwitz says that $2g-2=|N|R$. Since the $C$-action on $G/N$ lifts to a $C$-action on $G$, we see that the branch locus is permuted by $C$. Hence the branch locus is a disjoint union of $C$-orbits, each consisting of $p$ branch points with the same ramification. It follows that $R=pR'$ for some $R'\ge 1$. But then
$$
4(g-1)=2(2g-2)=2|N|R=2|N|pR' \ge 2\frac{k}{2}p=kp=|\Gamma|.
$$
Thus, in all cases we see that $|\Gamma|\le 4(g-1)$.
\end{proof}

\section{Sharpness of the Lower Bound}\label{LowerSharp}

We begin with a somewhat surprising proposition, which implies in particular that the function $M(g)$ only takes the values $4(g-1)$ and $6(g-1)$:

\begin{proposition}\label{2values}
Let $G$ be a graph of genus $g\ge 2$, and $\Gamma$ a group acting harmonically on $G$, with order $|\Gamma|>4(g-1)$. Then $|\Gamma|=6(g-1)$, i.e. the pair $(G,\Gamma)$ achieves the Hurwitz genus bound from section~\ref{Hurwitz}.
\end{proposition}

\begin{proof}
Note that the discussion in section~\ref{Hurwitz} implies that $g(G/\Gamma)=0$.
We proceed by determining the possibilities for the ramification in the map $G\rightarrow G/\Gamma$. Riemann-Hurwitz says that
\begin{eqnarray*}
2g-2=|\Gamma|(-2+R)
>(4g-4)(-2+R),
\end{eqnarray*}
which implies that
$$
\frac{5}{2}> R=\sum_{i=1}^s[2(1-\frac{1}{r_i})+w_i].
$$

Since the map must be ramified, we conclude that $s=1$ or $s=2$. Moreover, $g(G)\ge 2$ implies that $R>2$, so the map must have horizontal ramification. Consider the case of two branch points, one of which has only vertical ramification in the fiber. Then $2<R=2(1-\frac{1}{r_1})+w_1+w_2<\frac{5}{2}$ implies that 
$(r_1,r_2; w_1,w_2)=(3,1;0, 1)$. This corresponds to case 3 of Proposition~\ref{R>2}, and we see that $R=\frac{7}{3}$, so that $|\Gamma|=6(g-1)$ as claimed.
If both branch points have horizontal ramification, the condition $2<R=2(1-\frac{1}{r_1})+2(1-\frac{1}{r_2})+w_1+w_2<\frac{5}{2}$ implies that $(r_1,r_2;w_1,w_2)=(3,2;0,0)$ which is case 2 of Proposition~\ref{R>2}. Again, this implies that $|\Gamma|=6(g-1)$. Finally, the case of a single branch point yields case 1 of Proposition~\ref{R>2}. Hence, $R=\frac{7}{3}$ and $|\Gamma|=6(g-1)$ in all cases. 

\end{proof}

We now complete the proof of Theorem~\ref{Main} with the following proposition.

\begin{proposition} 
The lower bound in Theorem~\ref{Main} is sharp: there exist infinitely many $g$ such that $M(g)=4(g-1)$.
\end{proposition}

\begin{proof}
Fix $g_1>1$ and suppose that $p$ is a prime satisfying the conditions of Proposition~\ref{psmall}, $G$ is a graph of genus $g=p(g_1-1)+1$, and $\Gamma$ is a group acting harmonically on $G$. Then $|\Gamma|\le 4(g-1)$, for if not, then Proposition~\ref{2values} would imply that $|\Gamma|=6(g-1)=6p(g_1-1)$, and $p$ would divide $|\Gamma|$, contradicting Proposition~\ref{psmall}. 

So in the situation of Proposition~\ref{psmall}, we have $M(g)=4(g-1)$. I claim that there are infinitely many such $g$. Indeed, set $g_1=2$, so that we will be considering prime numbers $p$ and $g=p+1$. In this case the hypotheses  of Proposition~\ref{psmall} require $p>6$ and $(\frac{p-1}{2},6)=1$. That is, $\frac{p-1}{2}$ must be odd and not divisible by $3$. But this means that $\frac{p-1}{2}=6k\pm 1$, or $p-1=12k\pm 2$, or finally $p=12k-1$ (the number $12k+3$ is never prime and greater than 6). By Dirichlet's Theorem, there are infinitely many primes of the form $12k-1$, which means that there are infinitely many $g$'s of the form $12k$ such that $g-1=p$ is prime. For all such $g$'s we have $M(g)=4(g-1)$, and the lower bound is sharp. 
\end{proof}

\end{document}